\documentclass[10pt, a4paper]{amsart}
\usepackage{article}

\usepackage[latin1]{inputenc}
\usepackage{hypbmsec,lmodern,microtype}
\usepackage[ngerman,american]{babel}
\usepackage[T1]{fontenc}
\usepackage[numbers,sort&compress]{natbib}
\usepackage{textcomp}
\usepackage{acronym}
\usepackage{capt-of}

\usepackage{fancyhdr}

\usepackage{pgfplots}
\usepgfplotslibrary{fillbetween}

\usepackage{tikz}

\usepackage{verbatim}

\usepackage[many]{tcolorbox}
\usepackage{tikzpagenodes}
\usetikzlibrary{tikzmark}

\usetikzlibrary{babel,quotes,angles,patterns,intersections}
\usetikzlibrary{calc}
\usetikzlibrary{fillbetween}
\usetikzlibrary{arrows}

\tikzset{every picture/.append style={remember picture},
	na/.style={baseline=-0.6ex}}

\tikzset{
	ncbar angle/.initial=90,
	ncbar/.style={
		to path=(\tikztostart)
		-- ($(\tikztostart)!#1!\pgfkeysvalueof{/tikz/ncbar angle}:(\tikztotarget)$)
		-- ($(\tikztotarget)!($(\tikztostart)!#1!\pgfkeysvalueof{/tikz/ncbar angle}:(\tikztotarget)$)!\pgfkeysvalueof{/tikz/ncbar angle}:(\tikztostart)$)
		-- (\tikztotarget)
	},
	ncbar/.default=0.5cm,
}
\tikzset{round left paren/.style={ncbar=0.5cm,out=120,in=-120}}
\tikzset{round right paren/.style={ncbar=0.5cm,out=60,in=-60}}

\usepackage[linecolor=white,bordercolor=blue!50,backgroundcolor=blue!25]{todonotes}
\usetikzlibrary{calc,patterns,angles,quotes}

\title[An Angular Transformation of Triangles]{An Angular Transformation of Triangles}
\author[Vartziotis]{Dimitris Vartziotis}
\address{}
\urladdr{}
\email{dimitris.vartziotis@twt-gmbh.de}
\author[Bohnet]{Doris Bohnet}
\address{}
\urladdr{}
\email{dbohnet@htwg-konstanz.de}

\date{\today}

\begin{document}

\maketitle

\begin{abstract}
Triangles are everywhere in the virtual world. The surface of nearly every graphical object is saved as a triangular mesh on a computer. Light effects and movements of virtual objects are computed on the basis of triangulations. Besides computer graphics, triangulated surfaces (see Fig.~\ref{fig:example} for an example) are used for the simulations of physical processes, like heating or cooling of objects or deformations. The numerical method for these simulations is often the finite element method, whose accuracy depends on the quality of the triangulation (see \cite{lo2015finite} for an overview on finite element meshes). The quality of a triangle is generally determined by computing its proximity to an equilateral triangle. Namely, the triangle's inner angles should neither be too small nor too big (see \cite{Shewchuk02whatis,BabuskaAziz1976}) in order to obtain reliable numerical results. Therefore, one often improves the mesh quality before any simulation. \\
The fact that we require triangulations for accurate simulations is the main motivation for our occupation with triangle transformations.\footnote{The book \cite{vartziotis2018getme} provides a broad presentation of the idea of mesh smoothing methods based on geometric element transformations.}  We need a triangulation method that transforms each triangle into a more regular one. However, the transformation should not regularize a particular triangle too fast as this may inhibit that the regularity a neighboring triangles can achieve. At the same time, we would like to prove the efficacy of the transformation. a property often missed by the heuristic procedures used in practice.\\
Besides the practical motivation, the transformation itself exhibits interesting properties which can nicely be proved by basic mathematics.  
\end{abstract}
\def\dotMarkRightAngle[size=#1](#2,#3,#4){%
 \draw ($(#3)!#1!(#2)$) -- 
       ($($(#3)!#1!(#2)$)!#1!90:(#2)$) --
       ($(#3)!#1!(#4)$);
 \path (#3) --node[circle,fill,inner sep=.5pt]{} ($($(#3)!#1!(#2)$)!#1!90:(#2)$);
}
\section{Introduction}
We define a transformation which converts any non-degenerate triangle into an equilateral triangle if it is applied iteratively. The transformation -- although it has its offspring in an elementary geometrical construction -- can be analytically expressed. This fact facilitates the mathematical analysis of the transformation. We aim to understand the convergence properties of the transformation: \textit{Does the transformation convert any triangle into an equilateral triangle? How quickly does it converge?}\footnote{We would like to point to the recent article \cite{Nicollier2019} where Nicollier takes a similar approach and decribes geometric triangle transformations as dynamical systems.}\\
In practice, the efficacy of a method is often not proved; Therefore, we are especially interested in these questions.\footnote{In fact, this has been of special interest in our research activities for quite a time, see e.g. \cite{VartziotisHimpel2014,Vartziotis2015,VartziotisBohnet2016,Vartziotis2017}}\\
\begin{figure}
\includegraphics[width=\textwidth]{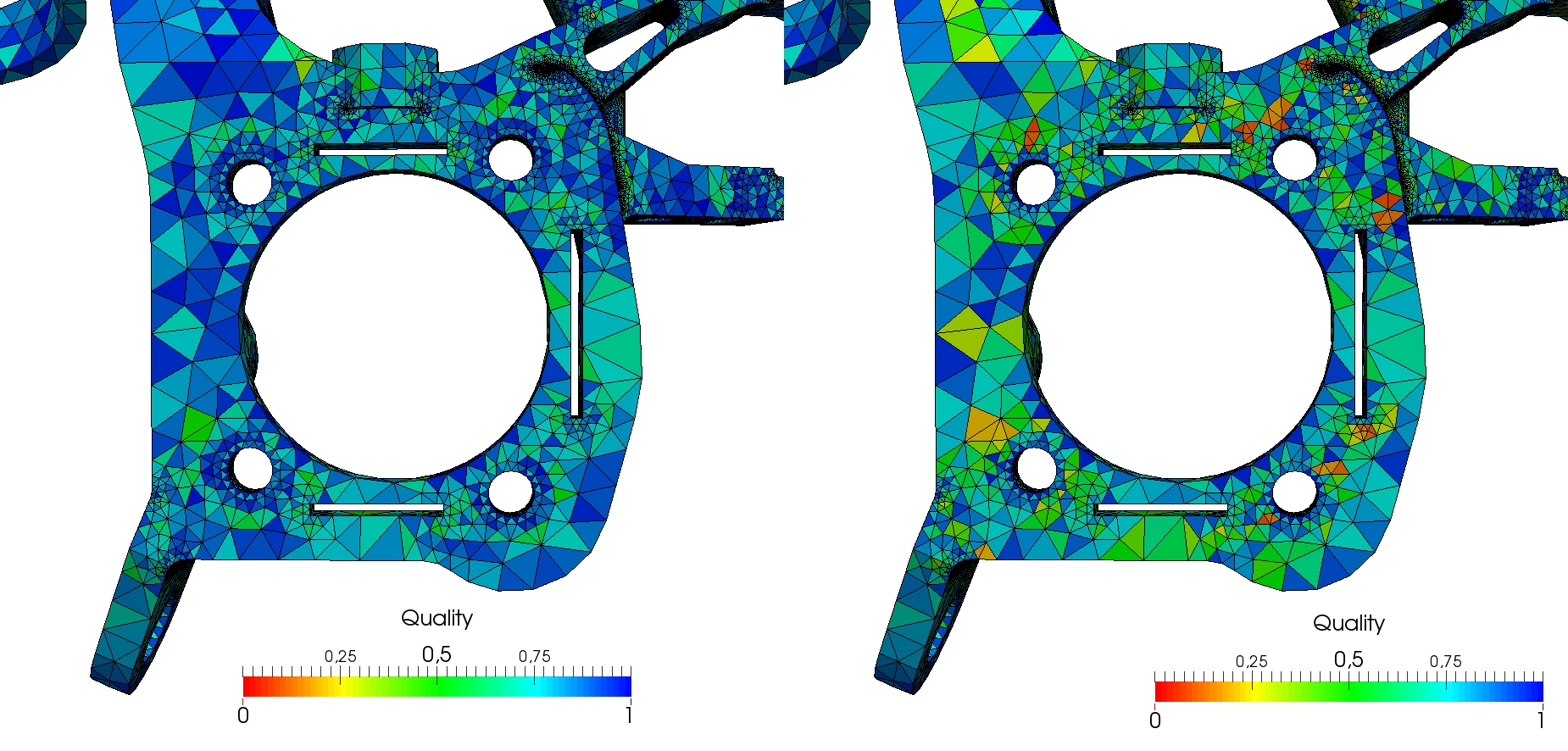}
\caption{Example of a triangulated planar surface. The color of each triangle depends on its quality: green/ blue triangles are quite close to equilateral ones, while orange/red triangles are rather distorted.}\label{fig:example}
\end{figure}
One noteworthy feature of the transformation is the fact that one can give an explicit formula for the quality of the triangle -- expressed as the ratio of the minimal and maximal inner angle -- in each iteration step. This allows, if implemented in order to smooth a triangle mesh, to foresee the quality of each triangle element and to decide which triangle should be smoothed depending on the resulting quality.\\
\begin{center}
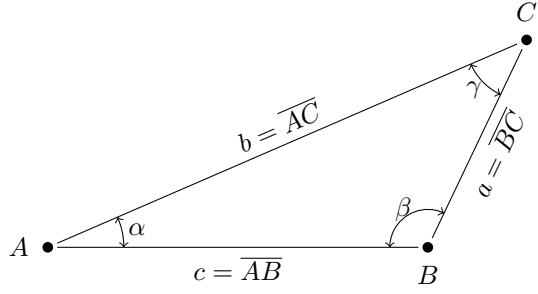

\begin{tikzpicture}

\node[label=left:$A$] (A) at (0,0) {};\node[label=below:$B$] (B) at (5,0) {};\node[label=above:$C$] (C) at (6.3,2.75) {};
\fill[black] (A) circle (2pt);
\fill[black] (B) circle (2pt);
\fill[black] (C) circle (2pt);
\draw (A) -- (B) node[midway,below] {$c=\overline{AB}$};
\draw (B)--(C) node[midway,sloped,below] {$a=\overline{BC}$};
\draw (C)--(A) node[midway,sloped,above] {$b=\overline{AC}$};
\path 
    (B)
    -- (A)
    -- (C)
  pic["$\alpha$",draw=black,<->,angle eccentricity=1.2,angle radius=1cm] {angle=B--A--C};
	\path 
    (C)
    -- (B)
    -- (A)
  pic["$\beta$",draw=black,<->,angle eccentricity=1.2,angle radius=0.5cm] {angle=C--B--A};
	\path 
    (A)
    -- (C)
    -- (B)
  pic["$\gamma$",draw=black,<->,angle eccentricity=1.2,angle radius=0.8cm] {angle=A--C--B};
\end{tikzpicture}
\captionof{figure}{\ Notations of a planar triangle}\label{fig:notations}
\end{center}
\section{Definition of the transformation}
We describe a procedure to construct a new triangle $\Delta'=(A'B'C')$ based on a triangle $\Delta=(ABC)$. The notations used throughout this article are shown in Fig.~\ref{fig:notations}. The construction is shown in Fig.~\ref{fig:firstIterate}. We require that the initial triangle $\Delta$ is non-degenerate, that is, every inner angle is strictly greater than 0 and smaller than $\pi$.  
\paragraph{\textbf{Triangle transformation:}}\label{proc}
\begin{enumerate}
\item Construct the bisecting lines of the inner angles of the triangles.
\item Construct three lines at the triangle nodes perpendicular to the bisecting lines. 
\item Define the three intersections points of these lines as the nodes $(A',B',C')$ of the new triangle $\Delta'$. 
\end{enumerate}
\begin{tikzpicture}

\node[label=left:$A$] (A) at (0,0) {};\node[label=below:$B$] (B) at (2.5,0) {};\node[label=above:$C$] (C) at (3.15,1.375) {};


\draw (A) --(B) ;
\coordinate  (o) at (1.76,0) ;
\coordinate (m) at (2.75,0.55);
\coordinate (n) at (1.95,0.85);
\draw (B)--(C) ;
\draw (C)--(A) ;
\fill[black] (A) circle (2pt);
\fill[black] (B) circle (2pt);
\fill[black] (C) circle (2pt);
\draw[dotted] (A)--(m);
\draw[dotted] (B)--(n);
\draw[dotted] (C)--(o);
\draw[name path=line1,dashed] ($(A)!3cm!270:(m)$) -- ($(A)!9cm!90:(m)$);
\draw[name path=line2,dashed] ($(B)!3.5cm!270:(n)$) -- ($(B)!3.5cm!90:(n)$);
\draw[name path=line3,dashed] ($(C)!9cm!270:(o)$) -- ($(C)!2.5cm!90:(o)$);
\fill[red, name intersections={of=line1 and line2}] (intersection-1) circle (2pt);
\coordinate (CC) at (intersection-1);

\node[label=right:$C'$] at (intersection-1) {} ;
\fill[red, name intersections={of=line2 and line3}] (intersection-1) circle (2pt);
\node[label=right:$A'$] at (intersection-1) {};
\coordinate (AA) at (intersection-1);
\fill[red, name intersections={of=line1 and line3}] (intersection-1) circle (2pt);
\node[label=right:$B'$] at (intersection-1) {};
\coordinate (BB) at (intersection-1);
\draw(CC)--(AA)--(BB)--(CC);
\dotMarkRightAngle[size=6pt](AA,A,m);
\dotMarkRightAngle[size=6pt](BB,B,n);
\dotMarkRightAngle[size=6pt](CC,C,o);
\path 
    (AA)
    -- (CC)
    -- (BB)
  pic["$\gamma'$",draw=black,<->,angle eccentricity=1.2,angle radius=0.5cm] {angle=AA--CC--BB};
	\path 
    (BB)
    -- (AA)
    -- (CC)
  pic["$\alpha'$",draw=black,<->,angle eccentricity=1.2,angle radius=0.5cm] {angle=BB--AA--CC};
	\path 
    (CC)
    -- (BB)
    -- (AA)
  pic["$\beta'$",draw=black,<->,angle eccentricity=1.2,angle radius=0.8cm] {angle=CC--BB--AA};
\end{tikzpicture}

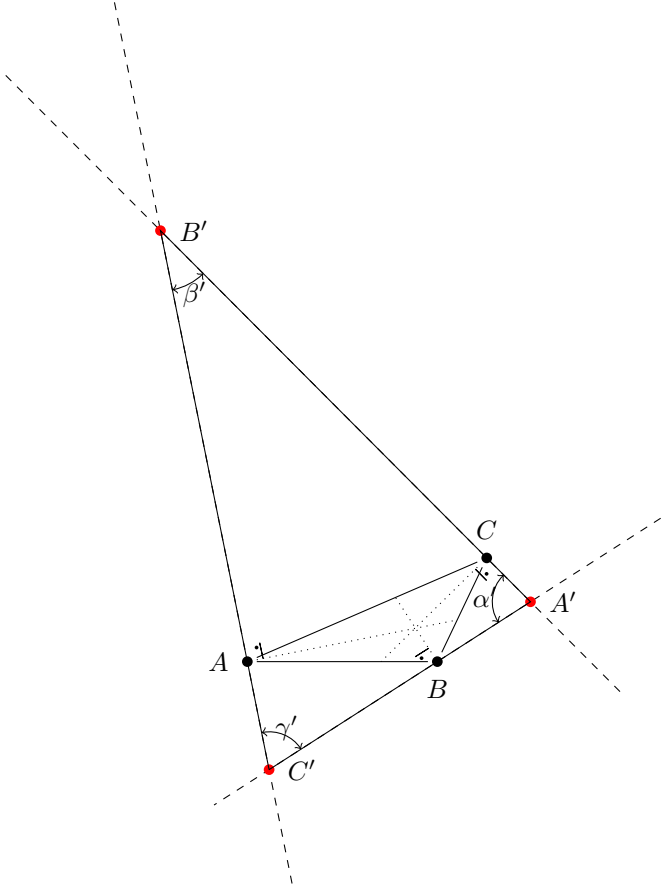
\captionof{figure}{ \ Construction of the new triangle $\Delta'=(A'B'C')$ from $\Delta=(ABC)$.}\label{fig:firstIterate}
The new triangle - seen above - has increased in size considerably; As such, we might want to rescale its dimensions. For the moment, we neglect the size of the triangle and consequently, we identify all triangles which are similar to each other. Similarity of triangles is an equivalence relation, so we can consider the equivalence class $[\Delta]$ of each triangle $\Delta$ which contains all triangles similar to $\Delta$. Similar triangles are uniquely characterized by their inner angles disregarding the order of the nodes. Each equivalence class $[\Delta]$ can then be identified by the inner angles $(\alpha,\beta,\gamma)$ of $\Delta$. Consequently, the set $\mathcal{T}/\sim$ of similar triangles can be identified with a subset of $\left(0,\pi\right)^3$:
$$\mathcal{T}/\sim := \left\{(\alpha,\beta,\gamma) \in (0,\pi)\times(0,\pi)\times (0,\pi)\big|\; \alpha+\beta+\gamma=\pi\right\}.$$
For simplicity, we denote an element of $\mathcal{T}/\sim$ by $\Delta$ itself, instead of $[\Delta]$. \\
We are interested in iterated applications of the procedure~(\ref{proc}). Does the procedure converge to a specific triangle?\\
In order to answer this question, we search for an analytical expression for the described triangle transformation. We would like to define a transformation $T$ on the set $\mathcal{T}/\sim $ of similar triangles. Therefore, we must express the inner angles $(\alpha',\beta',\gamma')$ of $\Delta'$ as a function of the inner angles $(\alpha,\beta,\gamma)$ of $\Delta$:\\
We look at the small triangle $AC'B$ (see Fig.~\ref{fig:firstIterate}) and denote its inner angles by $\theta_A,\theta_B$ and $\gamma'$. We can calculate the angles thanks to our construction procedure:
$$\theta_A=\frac{\pi-\alpha}{2},\quad \theta_B=\frac{\pi-\beta}{2}, \quad \gamma'=\pi-\theta_A-\theta_B=\frac{\alpha +\beta}{2}.$$
In the same way, we can compute the remaining inner angles with $$\beta'=\frac{\alpha+\gamma}{2},\quad \alpha'=\frac{\beta+\gamma}{2}.$$
The new triangle $\Delta'$ has inner angles that are the mean value of two inner angles of $\Delta$. So, we can define the transformation with   
$$T:\Delta \mapsto \Delta',\quad (\alpha,\beta,\gamma)\mapsto \left(\frac{\beta+\gamma}{2},\frac{\alpha+\gamma}{2},\frac{\alpha +\beta}{2}\right).$$ 
The first question is always if a newly defined transformation is well-defined:
\begin{thm}
The map $T$ is well-defined on the set $\mathcal{T}/\sim$ of similar triangles. If $\Delta$ is non-degenerate, $T(\Delta)$ is also non-degenerate.
\end{thm}
\begin{proof}
We have to show that $T(\Delta) \in \mathcal{T}/\sim$ for any $\Delta \in \mathcal{T}/\sim$: It suffices to prove that the angles $(\alpha',\beta',\gamma')$ of $\Delta'=T(\Delta)$ sum up to $\pi$: 
$$\alpha'+\beta'+\gamma'=\frac{\beta+\gamma}{2}+\frac{\alpha+\gamma}{2}+\frac{\alpha +\beta}{2}=\alpha+\beta + \gamma=\pi.$$
So $T(\Delta)$ is itself a triangle. \\
Since $0 <\alpha,\beta,\gamma<\pi$, then it follows directly that $0<\frac{\beta+\gamma}{2},\frac{\alpha+\gamma}{2},\frac{\alpha +\beta}{2}< \pi$. This means that the new triangle is non-degenerate. 
\end{proof}  
\begin{rem} The transformation $T$ is a linear transformation, so we can express it as a matrix:
$$T(\Delta)=\begin{pmatrix}0 &\frac{1}{2}&\frac{1}{2}\\\frac{1}{2}& 0 &\frac{1}{2}\\0 & \frac{1}{2}&\frac{1}{2}\end{pmatrix}\begin{pmatrix}\alpha\\ \beta \\\gamma\end{pmatrix}.$$
We call the matrix $T$ itself.
It has the double eigenvalue $-\frac{1}{2}$ and $1$. 
\end{rem}
\section{Convergence properties of the transformation}
The transformation $T$ can be applied iteratively to a triangle $\Delta$, that is, $$T^n(\Delta):=T\circ T \circ \dots \circ T\left(\Delta\right).$$ In this section, we look at this sequence of triangles $T^n(\Delta)$. First, we show that it converges to an equilateral triangle: 
\begin{thm}
For any non-degenerate triangle $\Delta = (\alpha,\beta,\gamma) \in \mathcal{T}/\sim$ we have
$$\lim_{n\rightarrow\infty}T^n\left(\Delta\right)=\left(\frac{\pi}{3},
\frac{\pi}{3},\frac{\pi}{3}\right).$$
\end{thm}
\begin{proof}
We use the matrix representation of the transformation $T$ and the fact that $T$ is diagonalizable in order to compute the following limit:
$$\lim_{n \rightarrow \infty} T^n\left(\Delta\right)=\lim_{n\rightarrow \infty}  \begin{pmatrix}0 &\frac{1}{2}&\frac{1}{2}\\\frac{1}{2}& 0 &\frac{1}{2}\\0 & \frac{1}{2}&\frac{1}{2}\end{pmatrix}^n\begin{pmatrix}\alpha\\ \beta \\\gamma\end{pmatrix}.$$
The matrix $T$ has a double eigenvalue $\lambda_{1,2}=-\frac{1}{2}$ with eigenvectors $v_1=(-1,1,0)$, $v_2=(-1,0,1)$, and eigenvalue $\lambda_3=1$ with $v_3=(1,1,1)$.
With the help of the matrix $$S=[v_1 v_2 v_3]=\begin{pmatrix}-1&-1&1\\1 &0&1\\0 & 1&1\end{pmatrix},$$ we can diagonalize the matrix $T$. Consequently, we can rewrite $T^n$ as:
\begin{align*}
\lim_{n\rightarrow\infty}\begin{pmatrix}0 &\frac{1}{2}&\frac{1}{2}\\\frac{1}{2}& 0 &\frac{1}{2}\\0 & \frac{1}{2}&\frac{1}{2}\end{pmatrix}^n\begin{pmatrix}\alpha\\ \beta \\\gamma\end{pmatrix}
&=\lim_{n\rightarrow \infty}S\begin{pmatrix}\left(-\frac{1}{2}\right)^n&0&0\\0& \left(-\frac{1}{2}\right)^n &0\\0 & 0&1\end{pmatrix}S^{-1}\begin{pmatrix}\alpha\\ \beta \\\gamma\end{pmatrix}\\
&= S\begin{pmatrix} 0 &0 & 0\\ 0 & 0 &0\\0 & 0 &1\end{pmatrix}S^{-1}\begin{pmatrix}\alpha\\ \beta \\\gamma\end{pmatrix}\\
&= \frac{1}{3}\begin{pmatrix} 1 & 1 &1 \\1 & 1 & 1\\1 & 1 & 1\end{pmatrix}\begin{pmatrix}\alpha\\ \beta \\\gamma\end{pmatrix}\\
&= \frac{1}{3}\begin{pmatrix} \alpha + \beta + \gamma\\\alpha+\beta + \gamma\\\alpha + \beta + \gamma\end{pmatrix}
\end{align*}
As $\alpha+\beta+\gamma=\pi$, we get $\lim_{n\rightarrow\infty}T^n\left(\Delta\right)=\left(\frac{\pi}{3},
\frac{\pi}{3},\frac{\pi}{3}\right)$ finishing the proof.
\end{proof}
This result implies directly that
\begin{corol}
The sequence $T^n(\Delta)$ of any non-degenerate triangle $\Delta$ -- up to similarity -- converges to an equilateral triangle.
\end{corol}
There is one shortcoming of the transformation $T$: the edge lengths increase as we have seen in Fig.~\ref{fig:firstIterate}, and we can also show this with the following theorem based on trigonometric laws: 
\begin{thm}
Let $\Delta$ be a non-degenerate triangle with edge lengths $a,b,c>0$, and $T$ the transformation defined in (\ref{proc}).\\
For any $n >0$ it holds that the edge lengths $a_n,b_n,c_n$ of $T^n(\Delta)$ fulfill the following formula: 
$$a_{3n}=a \prod_{j=0}^n f_j, \; b_{3n}=b \prod_{j=0}^nf_j, \; 
c_{3n}=c \prod_{j=0}^nf_j, $$
where $\alpha_j,\beta_j,\gamma_j$ denote the inner angle of the $j$th-iterate $T^j(\Delta)$ and 
$$f_j:=\frac{1}{\sin\left(\frac{\gamma_{j}}{2}\right)\sin\left(\frac{\beta_{j}}{2}\right)\sin\left(\frac{\alpha_{j}}{2}\right)}.$$
\end{thm}
\begin{proof}
We consider the small triangles $A'BA'$,$ACC'$ and $B'CB$ in $T(\Delta)$ (see Fig.~\ref{fig:smallTriangle}). Applying the sinus law, $$\frac{a}{b}= \frac{\sin(\alpha)}{\sin(\beta)},$$we obtain:
\begin{align*}
b'&=b_1+b_2\\
&= c\frac{\sin\left(\frac{\pi-\alpha}{2}\right)}{\sin\left(\gamma'\right)}+a\frac{\sin\left(\frac{\pi-\gamma}{2}\right)}{\sin\left(\alpha'\right)} \\
&= c\frac{\cos\left(\frac{\alpha}{2}\right)}{\cos\left(\frac{\gamma}{2}\right) }+ a \frac{\cos\left(\frac{\gamma}{2}\right)}{\cos\left(\frac{\alpha}{2}\right)}\quad\big|\cdot\frac{1}{a}, \; \frac{c}{a}=\frac{\sin(\gamma)}{\sin(\alpha)}.\\
\frac{b'}{a}&= \frac{\sin\left(\gamma\right)}{\sin\left(\alpha\right)}\frac{\cos\left(\frac{\alpha}{2}\right)}{\cos\left(\frac{\gamma}{2}\right)} + \frac{\cos\left(\frac{\gamma}{2}\right)}{\cos\left(\frac{\alpha}{2}\right)}\quad \mbox{mit} \; \sin(2x)=2\sin(x)\cos(x),\; x=\alpha,\gamma\\
&= \frac{\cos\left(\frac{\beta}{2}\right)}{\sin\left(\frac{\alpha}{2}\right)\cos\left(\frac{\alpha}{2}\right)}\\
&=2\frac{\cos\left(\frac{\beta}{2}\right)}{\sin\left(\alpha\right)}
\end{align*}
As this is true for every edge length, we obtain:
$$b_{3}= 2^3b \frac{\cos\left(\frac{\beta}{2}\right)}{\sin\left(\alpha\right)}\frac{\cos\left(\frac{\alpha}{2}\right)}{\sin\left(\beta\right)}\frac{\cos\left(\frac{\gamma}{2}\right)}{\sin\left(\gamma\right)}$$
Using $\sin(2x)=2\sin(x)\cos(x)$, we get:
$$b_3 = b \frac{1}{\sin\left(\frac{\alpha}{2}\right)\sin\left(\frac{\beta}{2}\right)\sin\left(\frac{\gamma}{2}\right)}.$$
As this is true for any iterate $n$ the statement is proved.
\end{proof}
\begin{rem}
We look at the factor $f_j$. As the angles $\frac{\alpha_j}{2},\frac{\beta_j}{2},\frac{\gamma_j}{2} \in \left(0,\frac{\pi}{2}\right)$, the factor is always strictly greater than $1$: $f_j >1$. Further, the angles approach eventually $\frac{\pi}{6}$, so that $f_j \rightarrow 2^3$. Therefore, the product $$\lim_{n\rightarrow \infty}\prod_{j=0}^n f_j=\infty$$
grows to infinity and so do the edge lengths $a_n,b_n,c_n$.  
\end{rem}

\begin{tikzpicture}

\node[label=left:$A$] (A) at (0,0) {};\node[label=below:$B$] (B) at (5,0) {};\node[label=above:$C$] (C) at (6.3,2.75) {};


\draw (A) --(B) ;
\coordinate  (o) at (3.55,0) ;
\coordinate (m) at (5.5,1.1);
\coordinate (n) at (3.9,1.7);
\draw (B)--(C) ;
\draw (C)--(A) ;
\fill[black] (A) circle (2pt);
\fill[black] (B) circle (2pt);
\fill[black] (C) circle (2pt);
\draw[dotted] (A)--(m);
\draw[dotted] (B)--(n);
\draw[dotted] (C)--(o);
\draw[name path=line1,dashed] ($(A)!6cm!270:(m)$) -- ($(A)!1cm!90:(m)$);
\draw[name path=line2,dashed] ($(B)!7cm!270:(n)$) -- ($(B)!7cm!90:(n)$);
\draw[name path=line3,dashed] ($(C)!1cm!270:(o)$) -- ($(C)!5cm!90:(o)$);
\fill[red, name intersections={of=line1 and line2}] (intersection-1) circle (2pt);
\coordinate (CC) at (intersection-1);

\node[label=right:$C'$] at (intersection-1) {} ;
\fill[red, name intersections={of=line2 and line3}] (intersection-1) circle (2pt);
\node[label=right:$A'$] at (intersection-1) {};
\coordinate (AA) at (intersection-1);
\coordinate (BB) at (intersection-1);
\draw(A)--(CC)--(AA)--(C);
\draw(A)--(B) node[midway,below] {$c$};
\draw(B)--(C) node[midway,below] {$a$};
\draw(CC)--(B) node[midway,below] {$a_1$};
\draw(B)--(AA) node[midway,below] {$a_2$};
\dotMarkRightAngle[size=6pt](AA,A,m);
\dotMarkRightAngle[size=6pt](BB,B,n);
\dotMarkRightAngle[size=6pt](CC,C,o);
\path 
    (A)
    -- (CC)
    -- (B)
  pic["$\gamma'$",draw=black,<->,angle eccentricity=1.2,angle radius=0.8cm] {angle=B--CC--A};
	\path 
    (B)
    -- (AA)
    -- (C)
  pic["$\alpha'$",draw=black,<->,angle eccentricity=1.2,angle radius=0.8cm] {angle=C--AA--B};

\end{tikzpicture}


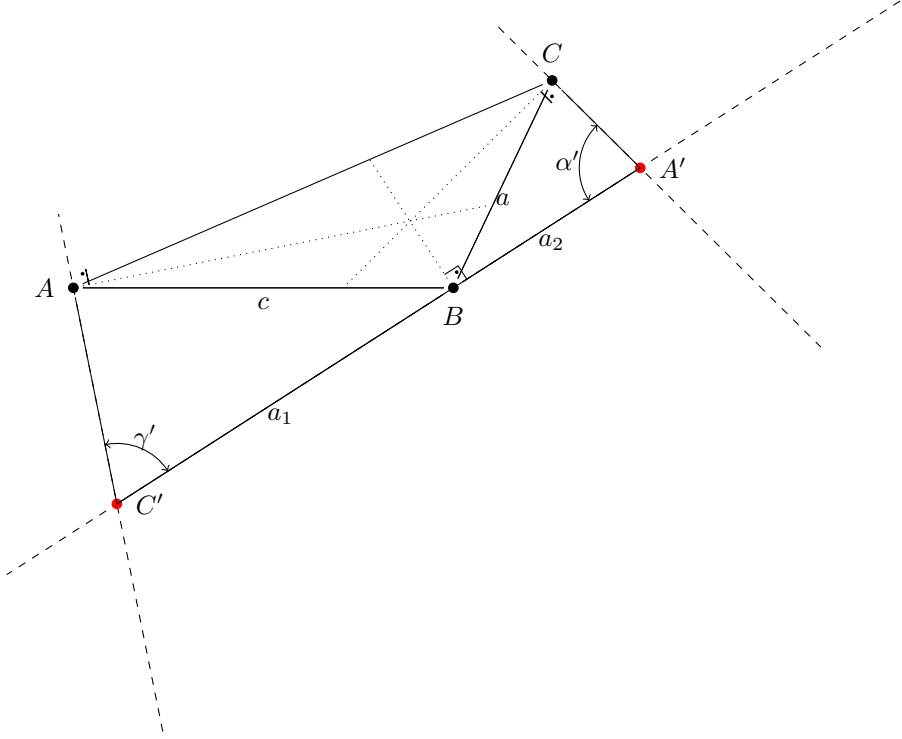
\captionof{figure}{ \ The edge lengths are increasing under the transformation.}\label{fig:smallTriangle}

\section{Quality function}
When dealing with meshes it is important to assess the quality of a mesh (see e.g. \cite{Knupp2007} for an overview on quality meshes). Usually, a quality function is a function comprised of all nodes and scaled from $[0,1]$. The value $1$ is attained from a mesh exclusively consisting of regular elements. A quality function should possess both scaling and rotational invariance.\\ 
For triangular meshes, the degree of regularity is often computed as a function of the inner triangle angles of each element. The accuracy of numerical methods relies on the absence of very small angles; Hence, this quality function is practical since all inner angles should be bounded away from zero. We follow this approach here: 
Let $q_n$ denote the quality function of $T^n(\Delta)$ defined by the ratio of the minimal and maximal inner angle: 
$$q_n:=\frac{\min\left\{\alpha_n,\beta_n,\gamma_n\right\}}{\max\left\{\alpha_n,\beta_n,\gamma_n\right\}}.$$
\begin{thm} Let the notations be as above, assume $\alpha_0\geq \beta_0 \geq \gamma_0$, then the following formula is true for $k\geq 0$: 
\begin{align}
q_{2k}&= \frac{\pi-b_{2k}\alpha_0}{\pi-b_{2k}\gamma_0}, \quad q_{2k+1}&= \frac{\pi-b_{2k+1}\alpha_0}{\pi-b_{2k+1}\gamma_0}
\end{align}
where $b_{2k} =\frac{3}{4^k-1}$ and $b_{2k+1}=\frac{6}{4^{n+1}+2}$.
\end{thm}
\begin{proof}
First of all, one observes, that the order of the angles is 2-periodically exchanged:
We assume $\alpha_0\geq \beta_0 \geq \gamma_0$. After the first iteration, we get 
$$\frac{\gamma_0 + \beta_0}{2} \leq \frac{\alpha_0+\gamma_0}{2}\leq \frac{\alpha_0+\beta_0}{2},\quad\mbox{that is}, \; \alpha_1\leq \beta_1\leq \gamma_1.$$
Consequently, we have recovered the original order of the angles after two iterations.\\
We have defined the transformation $T$ as a matrix transformation of the angles with
$$T\left(\alpha,\beta_0,\gamma_0)\right)=\begin{pmatrix}0 &\frac{1}{2}&\frac{1}{2}\\ \frac{1}{2}&0&\frac{1}{2}\\\frac{1}{2}&\frac{1}{2}&0\end{pmatrix}\begin{pmatrix}\alpha\\\beta\\\gamma\end{pmatrix}.$$
One computes the $n$th iterate with
$$T^n\left(\alpha,\beta_0,\gamma_0)\right)=\begin{pmatrix}\frac{a_{n-1}}{2^{n-1}} &\frac{a_{n}}{2^{n}} &\frac{a_{n}}{2^{n}} \\ \frac{a_{n}}{2^{n}} &\frac{a_{n-1}}{2^{n-1}} &\frac{a_{n}}{2^{n}} \\\frac{a_{n}}{2^{n}} &\frac{a_{n}}{2^{n}} &\frac{a_{n-1}}{2^{n-1}} \end{pmatrix}\begin{pmatrix}\alpha\\\beta\\\gamma\end{pmatrix},$$
where $a_n=a_{n-1}+2a_{n-2}$ for $n\geq 3$ and $a_1=a_2=1$. 
Therefore, we have 
\begin{align}\label{eq:angleFormulas}
\alpha_n &= \frac{a_{n-1}}{2^{n-1}} \alpha_0 + \frac{a_{n}}{2^{n}}\left(\beta_0 + \gamma_0\right) \nonumber\\
&= \frac{1}{2^{n}}\left(\left(2a_{n-1} -a_{n}\right)\alpha_0 + a_n \pi\right)\\
\beta_n &= \frac{a_{n-1}}{2^{n-1}} \beta_0 + \frac{a_{n}}{2^{n}}\left(\alpha_0 + \gamma_0\right) \nonumber\\
&= \frac{1}{2^{n}}\left(\left(2a_{n-1} -a_{n}\right)\beta_0 + a_n \pi\right)\nonumber\\
\gamma_n &= \frac{a_{n-1}}{2^{n-1}} \gamma_0 + \frac{a_{n}}{2^{n}}\left(\alpha_0 + \beta_0\right) \nonumber\\
&= \frac{1}{2^{n}}\left(\left(2a_{n-1} -a_{n}\right)\gamma_0 + a_n \pi\right)\nonumber
\end{align}
We compute the coefficients $a_n$ for every $n$th iterate:
\begin{lemma} For $n\geq 2$ and the notations as above we have:
\begin{align}\label{eq:coefficientsQuality}
a_{2n} &= \sum_{j=0}^{n-1}4^j=\frac{4^n-1}{3}\\
a_{2n+1}&=\frac{1}{2}\left(1+\sum_{j=0}^{n}4^j\right) = \frac{4^{n+1}+2}{6}
\end{align}
\end{lemma}
\begin{proof}[Lemma]
For $n=2$, we insert the definition of the coefficient $a_{2n}$ and get: $a_4=a_3+2a_2=a_2+2a_1+2a_2=5$. This is equal to $\frac{4^2-1}{3}$. 
For $n=2$, we get for odd coefficients $a_5=a_4+2a_3=11$ and this is equal to $\frac{4^3+2}{6}$.
Let the formula be true for $n\geq 2$. Then we compute
\begin{align*}
a_{2n+2} &=a_{2n+1}+2a_{2n}\\
&= \frac{4^{n+1}+2}{6} + 2 \frac{4^n-1}{3}\\
&= \frac{4^{n+1}+2+4^{n+1}-4}{6}\\
&=\frac{4^{n+1}-1}{3} \quad \mbox{and for odd coefficients:}\\
a_{2n+3} &=a_{2n+2}+2a_{2n+1}\\
&= \frac{4^{n+1}-1}{3} + 2 \frac{4^{n+1}+2}{6}\\
&= \frac{4^{n+2}+2}{6}
\end{align*}
finishing the proof.
\end{proof}
If we insert the coefficients $a_n$ into our formulas Eq.~\ref{eq:angleFormulas}, we obtain the statement and finish our proof. 
\end{proof}
\begin{rem}
If $n\rightarrow \infty$, we obtain $q_k\rightarrow 1$ as expected.
\end{rem}
One should note that we have obtained a non-recursive formula to compute the triangle quality in each iteration step. In this way one can foresee the triangle quality which can be very helpful in practical applications.
\begin{corol}
The convergence speed of the transformation $T^2$ is linear with a convergence rate of $\frac{1}{4}$.
\end{corol}
\begin{proof}
One computes with the formula~\ref{eq:angleFormulas} above that 
\begin{align*}
\alpha_{2n} - \frac{\pi}{3}&= \frac{1}{2^{2n}}\left( (2a_{2n-1}-a_{2n})\alpha_0 + a_{2n}\pi\right)-\frac{\pi}{3}\\
&= \frac{1}{4^n}\left(\alpha_0 -\frac{\pi}{3}\right)
\end{align*}
And consequently, one gets
$$\frac{\left|\alpha_{2n} - \frac{\pi}{3}\right|}{\left|\alpha_{2n-2} - \frac{\pi}{3}\right|}=\frac{1}{4}$$
finishing the proof.
\end{proof}
\section{Transformation for simple meshes}
We have studied the triangle transformation with the objective of mesh smoothing. For this purpose, one should rescale the triangles after each transformation such that the area is kept constant. In the following, we would like to give a small insight into the practical applications of triangle transformation. For this reason, we look at a so-called simple mesh which we can regularize with our slightly adapted transformation. \\
We call a triangle mesh a \textbf{$N$-simple mesh} if it consists of one inner vertex and $N$ boundary vertices that form a polygon with $N$ vertices. We denote the inner vertex as $A_0$ and the boundary vertices as $A_1,\dots,A_N$, and the $N$ triangles by $\Delta_1,\dots,\Delta_N$- see Fig.~\ref{fig:simpleMesh} for notations. The inner angle of triangle $\Delta_i$ at the inner vertex is denoted by $\alpha_i$. The triangle angles at the boundary nodes are denoted as $\beta_i$ and $\gamma_i$, respectively. If we apply the transformation to the triangles of a simple mesh we have to guarantee that the connectivity of the mesh is kept. In other words, the following equations must be fulfilled in any iteration step $n\geq 0$:
\begin{align}\label{eq:constraints}
&\mbox{Sum of inner angles of all triangles:} \quad\alpha_i + \beta_i + \gamma_i =\pi, \quad i=1,\dots,N\\
&\mbox{Sum of half of inner angles of polygon:}\quad \sum_{i=1}^N\beta_i=\frac{(N-2)}{2}\pi.\\
&\mbox{Sum of inner angles at the inner vertex $A_0$:}\quad \sum_{i=1}^N\alpha_i = 2\pi.
\end{align}
Consequently, we correct the triangle transformation defined above by a correction term $K(\alpha),K(\beta),K(\gamma_i)$: 
\begin{align*}
\alpha_i' &= \frac{\beta_i + \gamma_i}{2}+K(\alpha)\\
\beta_i'&=\frac{\alpha_i+\gamma_i}{2} + K(\beta)\\
\gamma_i'&=\frac{\alpha_i+\beta_i}{2} + K(\gamma)
\end{align*}
With the help of the equations~(\ref{eq:constraints}) we compute
\begin{align*}
\sum_i \alpha_i' + NK(\alpha) &= 2\pi\\
\sum_i \frac{\beta_i+\gamma_i}{2} + NK(\alpha)&=2\pi\\
 K(\alpha)&=\frac{1}{N}\left(2\pi - \frac{N-2}{2}\pi\right)\\
K(\alpha)&=\frac{\pi}{2N}\left(6-N\right)
\end{align*}
\begin{align*}
\sum_i \beta'_i+N(K(\beta) &= \frac{(N-2)}{2}\pi\\
\sum_i \frac{\alpha_i+\gamma_i}{2} + N(K(\beta) &=\frac{(N-2)}{2}\pi\\
K(\beta) &= \frac{\pi}{4N}\left(N-6\right)
\end{align*}
\begin{align*}
\alpha'_i+\beta'_i+\gamma'_i &= \pi\\
\alpha_i+\beta_i+\gamma_i + K(\alpha)+K(\beta)+K(\gamma)&=\pi\\
K(\alpha)+K(\beta)+K(\gamma)&=0
\end{align*}
We set $K(\beta)=K(\gamma)$ and $K(\gamma)=-\frac{1}{2}K(\alpha)$. 
\begin{rem}
If $N=6$, the correction terms vanish. The reason is that we can build a hexagon out of six equilateral triangles such that we do not need to correct our transformation. 
\end{rem}
We call the mesh quality $q_M$ of a $N$-simple mesh \textbf{optimal} if and only if each inner triangle $\Delta_i$ has an inner angle $\alpha_i=\frac{2\pi}{N}$ and angles $\beta_i=\gamma_i=\frac{(N-2)\pi}{2N}$, that is, each triangle $\Delta_i$ has a triangle the quality
$$q_i =\begin{cases} \frac{N-2}{4},\quad N&<6\\
\frac{4}{N-2}, \quad N&\geq 6\end{cases}.$$
We define the mesh quality by $$q_M = \frac{\min(q_i)}{\max(q_i)}.$$ The mesh is optimal, iff $q_M=1$. 
\begin{thm}
Let $\alpha_i,\beta_i,\gamma_i$ be the inner angles of a $N$-simple mesh as in Fig.~\ref{fig:simpleMesh}.Then 
the transformation $T_M$ defined by  
$$T_M(\alpha_i,\beta_i,\gamma_i)=\frac{1}{2}\begin{pmatrix} 0 & 1&1\\ 1 & 0&1\\1&1&0\end{pmatrix}+ \begin{pmatrix} \frac{\pi}{2N}(6-N)\\\frac{\pi}{4N}(N-6)\\\frac{\pi}{4N}(N-6)\end{pmatrix}$$
optimizes the mesh quality of any non-degenerate $N$-simple mesh if applied iteratively. 
\end{thm}
\begin{proof}
We start with a preliminary computation: 
\begin{align*}
\alpha_i^{(n)} &= \frac{\alpha_i^{(n-2)}}{4} + \frac{3}{2N}\pi\\
\alpha_i^{(2n)} &= \frac{\alpha_i}{4^n} + \frac{3}{2N}\pi\sum_{k=0}^{n-1}\frac{1}{4}^k\\
\lim_{n\rightarrow \infty} \alpha^{(2n)}_i &= \frac{3}{2N}\pi\cdot \frac{4}{3} = \frac{2}{N}\pi.
\end{align*}
For the inner angles $\beta_i$ and $\gamma_i$, the computation is identical: 
\begin{align*}
\beta_i^{(n)} &= \frac{\beta_i}{4^n}+\frac{(3N-2)\pi}{2N}\sum_{k=0}^{n-1}\frac{1}{4^k}\\
\lim_{n\rightarrow \infty}\beta_i^{(n)} &= \frac{(3N-2)\pi}{2N}\cdot \frac{4}{3} = \frac{(N-2)\pi}{2N}\\
\end{align*}
The quality of each triangle $\Delta_i$ of the limit simple $N$-mesh is then easily computed as 
$$q_i = \begin{cases} \frac{4}{N-2},\quad N\geq 6\\ \frac{N-2}{4}, \quad N<6\end{cases}.$$
This finishes the proof. 
\end{proof}
\begin{minipage}[t]{0.5\textwidth}
\begin{tikzpicture}

\node[label=left:$A_1$] (A) at (0,0) {};\node[label=below:$A_2$] (B) at (5,0) {};\node[label=above:$A_3$] (C) at (5,5) {};
\node[label=left:$A_4$] (D) at (0,5) {};\node[label=below:$A_0$] (E) at (3.4,1.7) {};


\draw[dotted] (A) --(E)--(B)--(A) ;
\draw[dotted] (A) --(E)--(D)--(A) ;
\draw[dotted] (B) --(E)--(C)--(B) ;
\draw[dotted] (D) --(E)--(C)--(D) ;
\draw (A)--(B)--(C)--(D)--(A);
\node[label=left:$\Delta_1$] (dreieck) at (2.8,0.57) {};
\fill[black] (A) circle (2pt);
\fill[black] (B) circle (2pt);
\fill[black] (C) circle (2pt);
\fill[black] (D) circle (2pt);
\fill[black] (E) circle (2pt);

\path 
    (B)
    -- (A)
    -- (E)
  pic["$\gamma$",draw=black,<->,angle eccentricity=1.2,angle radius=0.8cm] {angle=B--A--E};
	\path 
    (E)
    -- (B)
    -- (A)
  pic["$\beta$",draw=black,<->,angle eccentricity=1.2,angle radius=0.8cm] {angle=E--B--A};
\path 
    (A)
    -- (E)
    -- (B)
  pic["$\alpha$",draw=black,<->,angle eccentricity=1.2,angle radius=0.8cm] {angle=A--E--B};

\end{tikzpicture}

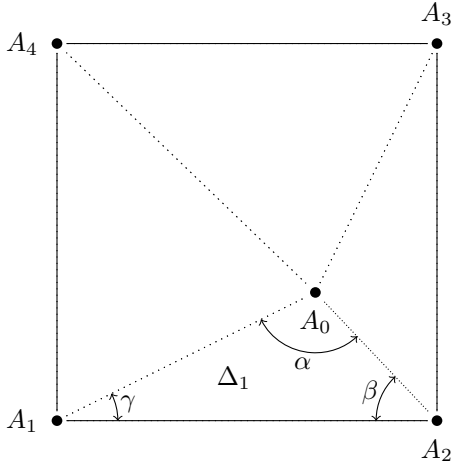
\captionof{figure}{ \ Simple $4$-mesh with notations.}\label{fig:simpleMesh}
\end{minipage}
\begin{minipage}[t]{0.50\textwidth}
\begin{tikzpicture}
\node[label=left:$A_1$] (A) at (0,0) {};\node[label=below:$A_2$] (B) at (5,0) {};\node[label=above:$A_3$] (C) at (5,5) {};
\node[label=left:$A_4$] (D) at (0,5) {};\node[label=below:$A_0$] (E) at (3.4,1.7) {};
\node[label=left:$A_1'$] (AA) at (0.1,-3.2) {};\node[label=below:$A_2'$] (BB) at (7,-0.9) {};\node[label=above:$A_4'$] (DD) at (-1.2,4) {};
aneu=(0.1cm,-3.2cm);
bneu=(7cm,-0.9cm);
dneu=(-1.2cm,4cm);
\draw[dotted] (A)--(E)--(B)--(A) ;
\draw[dotted] (A)--(E)--(D)--(A) ;
\draw[dotted] (B)--(E)--(C)--(B) ;
\draw[dotted] (D)--(E)--(C)--(D) ;
\draw (AA) --(E)--(BB)--(AA) ;
\draw (AA) --(E)--(DD)--(AA) ;
\draw (BB) --(E)--(C)--(BB) ;
\draw (DD) --(E)--(C)--(DD) ;
\node[label=left:$\Delta_1'$] (dreieck) at (3.5,-0.8) {};
\fill[black] (A) circle (2pt);
\fill[black] (B) circle (2pt);
\fill[black] (C) circle (2pt);
\fill[black] (D) circle (2pt);
\fill[black] (E) circle (2pt);
\fill[black] (AA) circle (2pt);
\fill[black] (BB) circle (2pt);
\fill[black] (DD) circle (2pt);
\path 
    (BB)
    -- (AA)
    -- (E)
  pic["$\gamma'$",draw=black,<->,angle eccentricity=1.2,angle radius=0.8cm] {angle=BB--AA--E};
	\path 
    (E)
    -- (BB)
    -- (AA)
  pic["$\beta'$",draw=black,<->,angle eccentricity=1.2,angle radius=0.8cm] {angle=E--BB--AA};
\path 
    (AA)
    -- (E)
    -- (BB)
  pic["$\alpha'$",draw=black,<->,angle eccentricity=1.2,angle radius=0.8cm] {angle=AA--E--BB};

\end{tikzpicture}

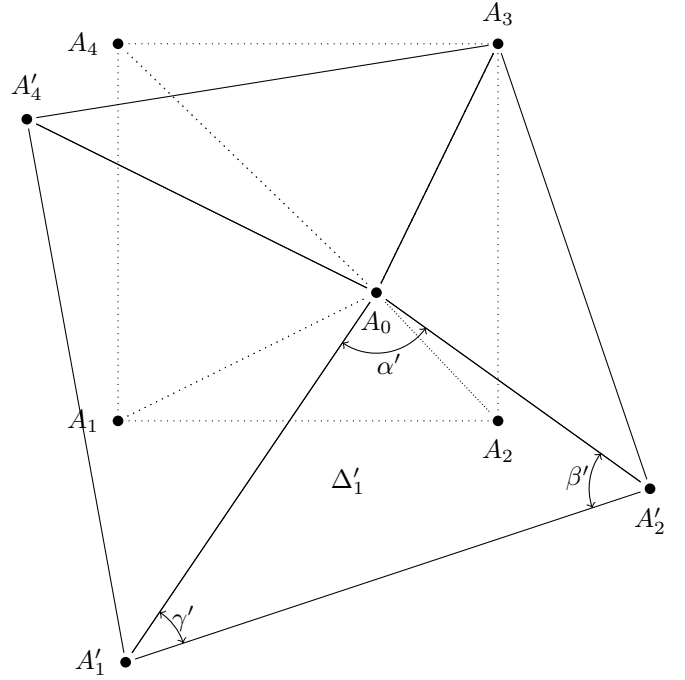
\captionof{figure}{ \ Iterated simple $4$-mesh. The dotted lines indicate the original simple mesh}\label{fig:iteratedSimpleMesh}
\end{minipage}
\section{Conclusions \& Outlook}
We have presented a triangle transformation that is derived from an easy geometric construction, analyzed its convergence properties, and given a simple example for its application on triangle meshes. \\
The present triangle transformation has attracted our interest because of this property: one can express the quality of the transformed triangle with a non-recursive formula. In order to implement this property in practice, we suggest the idea of a game-theoretical approach to mesh smoothing: one models each triangle in a triangle mesh as a \textit{player} and the improvement of the triangle quality as the payoff for each player. The decision of which triangle should be smoothed and how often this is being implemented is based on game theory. In this context, we have done numerical tests to explore if the present transformation also has convincing properties in practice -- if used for real mesh smoothing. These tests have not yet been completed.\\
Furthermore, the transformation can also be quite easily generalized to polygons and to any 2-dimensional polygonal mesh. However, the formulas get a bit fuzzy so we have not included this generalization into the present article.

\section*{Acknowledgements}
We thank Dylan Ascencios (Harvard University and TWT GmbH Science \& Innovation) and Dr. Florian Schneider (TWT GmbH Science \& Innovation) for thoroughly reviewing the paper.

\bibliographystyle{abbrv}
\bibliography{bibfile}{}
\bigskip

\bigskip
\end{document}